\documentclass{amsart}
\usepackage{amssymb,amsmath,amsthm,graphics,amscd,amsfonts}
\usepackage{color}

\theoremstyle{plain}
\newtheorem{theorem}{Theorem}[section]
\newtheorem{proposition}[theorem]{Proposition}

\newtheorem{lemma}[theorem]{Lemma}

\newtheorem{example}[theorem]{Example}
\newtheorem{definition}[theorem]{Definition}

\newcommand{\bfo}{{\bf o}}

\newcommand{\bfC}{{\mathbb C}}

\newcommand{\bfR}{{\mathbb R}}
\newcommand{\bfZ}{{\mathbb Z}}

\newcommand{\bfQ}{{\mathbb Q}}

\newcommand{\mapright}[1]{\smash{\mathop{   \hbox to 0.7cm{\rightarrowfill}}
  \limits^{#1}}}

\newcommand{\ArrowF}{{\overrightarrow{F}}}

\def\la{\langle}
\def\ra{\rangle}

\begin{document}

\title{Self-similar solutions to the mean curvature flows on Riemannian cone manifolds and special Lagrangians on
toric Calabi-Yau cones}
\author{Akito Futaki}
\address{Department of Mathematics, Tokyo Institute of Technology, 2-12-1,
O-okayama, Meguro, Tokyo 152-8551, Japan}
\email{futaki@math.titech.ac.jp}

\author{Kota Hattori}
\address{Department of Mathematics, Tokyo Institute of Technology, 2-12-1, O-okayama,
Meguro, Tokyo 152--8551, Japan} 
\email{kota-hattori@math.titech.ac.jp}

\author{Hikaru Yamamoto}
\address{Department of Mathematics, Tokyo Institute of Technology, 2-12-1,
O-okayama, Meguro, Tokyo 152-8551, Japan}
\email{yamamoto.h.ah@m.titech.ac.jp}

\date{June 4, 2012 }

\begin{abstract} 
The self-similar solutions to the mean curvature flow have been defined and studied on the Euclidean space. 
In this paper we propose a general treatment of the self-similar solutions to the mean curvature flow on Riemannian cone manifolds.
As a typical result we extend the well-known result of Huisken about the asymptotic behavior for the singularities of
the mean curvature flows. We also extend results on special Lagrangian submanifolds on $\mathbb C^n$ to
the toric Calabi-Yau cones over Sasaki-Einstein manifolds.
\end{abstract}

\maketitle

\section{Introduction}

Let $F:M\times [0,T)\rightarrow V$ be a smooth family of immersions of  an $m$-dimensional manifold $M$ 
into a Riemannian manifold $(V, g)$ of dimension $m+k$. 
$F$ is called a mean curvature flow if it satisfies 
\begin{equation}\label{mcfc1}
\begin{aligned}
& \frac{\partial F}{\partial t}(p,t)=H_{t}(p) &&\mathrm{~for~all~}(p,t)\in M\times[0,T)
\end{aligned}
\end{equation}
where $H_t$ is the mean curvature of the immersion $F_t := F(\cdot, t) : M \to V$. 

When $V$ is the Euclidean space $\bfR^{m+k}$ there is a well-studied
important class of solutions of (\ref{mcfc1}), that is, self-similar solutions. They are immersions $F : M \to \bfR^{m+k}$ 
satisfying
\begin{equation}\label{mcfc2} 
H = \lambda F^\perp
\end{equation}
where $\lambda$ is a constant and $F^\perp$ denotes the normal part of the position vector $F$. The solution of (\ref{mcfc2}) is called shrinking, stationary (or minimal) or expanding
depending on whether  $\lambda < 0$, $\lambda = 0$ or $\lambda > 0$. 

The purpose of this paper is to extend the definition of the self-similar solutions from the case when $V$ is the Euclidean spaces to 
the case when $V$ is a Riemannian cone manifold.
Let $(N,g)$ be an $n$-dimensional Riemannian manifold. We define the Riemannian cone manifold 
$(C(N),\overline{g})$ over $(N,g)$ by $C(N)=N\times\mathbb{R}^{+}$ and $\overline{g}=dr^2+r^2g$ where $r$ is the 
standard coordinate of $\mathbb{R}^{+}$. If $F : M \to C(N)$ is an immersion we define the
position vector $\ArrowF$ of $F$ at $p \in M$ by
\begin{align}\label{pv}
\ArrowF(p)= r(F(p))\frac{\partial}{\partial r}~~\in T_{F(p)}C(N). 
\end{align}
Then the self-similar solution is defined as
\begin{equation}\label{mcfc3} 
H = \lambda \ArrowF^\perp
\end{equation}
where $\lambda$ is a constant and $\ArrowF^\perp$ denotes the normal part of the position vector $\ArrowF$. 
In this paper we propose a general treatment of the self-similar solutions to the mean curvature flows on Riemannian cone manifolds.
As a typical result we extend the well-known result of Huisken about the asymptotic behavior for the singularities of
the mean curvature flows. 
In \cite{Huisken} Huisken introduced the rescaling technique and the monotonicity formula for the mean curvature flow 
of hypersurfaces in Euclidean space. Also in \cite{Huisken}, using the monotonicity formula, Huisken proved that 
if the mean curvature flow has the type I singularity then there exists a smoothly convergent subsequence of the rescaling 
such that its limit satisfies the self-similar solution equation. In this paper we extend those techniques and consequences to 
Riemannian cone manifolds and an initial date manifold. We also give a construction of self-similar solutions on Riemannian cone manifolds.

Let us recall the definition of type I singularity and its parabolic rescaling.
Let $M$ be a manifold and $(V,g)$ a Riemannian manifold. 
Suppose $F:M\times[0,T)\rightarrow V$ is a mean curvature flow with maximal time $T<\infty$ of existence of the solution. 
One says that $F$ develops a singularity of Type I as $t \to T$ if there exists a constant $C > 0$ such that 
\begin{align}
\sup _{M}\vert\mathrm{II}_{t}\vert^2\leq\frac{C}{T-t}~\mathrm{for~all~}t\in[0,T), \notag
\end{align}
where $\mathrm{II}_{t}$ is the second fundamental form with respect to the immersion $F_{t}:M\rightarrow V$. 
Otherwise one says that $F$ develops a singularity of Type II. 

Let $M$ be a manifold and $(C(N),\overline{g})$ the Riemannian cone manifold over a Riemannian manifold $(N,g)$. 
Take a constant $\lambda>0$. For a map $F:M\times[0,T)\rightarrow C(N)$, we define the parabolic rescaling of $F$ of scale $\lambda$ as 
follows;
\begin{align}
&F^{\lambda}:M\times[-\lambda^2 T,0)\rightarrow C(N);\notag\\
&F^{\lambda}(p,s)=(\pi_N(F(p,T+\frac{s}{\lambda^2})),\lambda r(F(p,T+\frac{s}{\lambda^2}))) \notag
\end{align}
where $\pi_N : C(N) = N \times \bfR^+ \to N$ is the standard projection.

When the singularity does not occur at the apex of the cone one can show that the parabolic rescaling of type I singularity gives rise to a
self-similar solution as shown by Huisken. However when the singularity occurs at the apex we need some more conditions. Thus we are
lead to the following definition of  type ${\mathrm I}_c$ singularity. 

\begin{definition}\label{typeI_c}
Let $M$ be a manifold and $(N,g)$ a Riemannian manifold. 
Suppose $F:M\times[0,T)\rightarrow C(N)$ is a mean curvature flow with $T<\infty$.
We say that $F$ develops a singularity of type $\mathrm{I}_c$ if the following three conditions are satisfied:
\begin{enumerate}
\item$F$ develops a singularity of type $\mathrm{I}$ as $t \to T$,\\
\item$r(F_t(p)) \to 0$ for some $p \in M$ as $t \to T$ and\\
\item$K_1(T-t) \leq \min_{ M}r^2(F_{t}) \leq K_2(T-t)$ for all $t \in [0,T)$
where  $K_1$ and $K_2$ are 
positive constants. 
\end{enumerate}
\end{definition}
\noindent
Examples of type $\mathrm I_c$ singularities are given in Example \ref{example}. 
\begin{theorem}\label{main}
Let $M$ be an $m$-dimensional compact manifold and $C(N)$ the Riemannian cone manifold over an $n$-dimensional Riemannian manifold $(N,g)$. 
Let $F: M\times[0,T)\rightarrow C(N)$ be a mean curvature flow, and assume that $F$ develops a type $\mathrm{I}_c$ singularity at $T$. 
Then, for any increasing sequence $\{\lambda_{i}\}_{i=1}^{\infty}$ of the scales of parabolic rescaling
such that $\lambda_{i}\rightarrow\infty$ as $i\rightarrow\infty$, 
there exist a subsequence $\{\lambda_{i_{k}}\}_{k=1}^{\infty}$ and a sequence $t_{i_k} \to T$ 
such that the sequence of rescaled mean curvature flow $\{F^{\lambda_{i_{k}}}_{s_{i_k}}\}_{k=1}^{\infty}$ 
with $s_{i_k} = \lambda^2_{i_k} (t_{i_k} - T)$ converges to a self-similar solution 
$F^{\infty}:M_\infty \rightarrow C(N)$ to the 
mean curvature flow. 
\end{theorem}

The proof of this theorem is not substantially different from Huisken's original proof. But the merit of the idea to study on cones will be that
we obtain examples of more non-trivial topology. In fact $N \cong \{r = 1\}$ in $C(N)$ is already a self-shrinker. Thus, any compact manifold
can be a self-shrinker in some Riemannian cone manifold. It is also possible to 
study special Lagrangian submanifolds and Lagrangian self-similar solutions in Calabi-Yau cones over Sasaki-Einstein manifolds.
A Sasaki manifold $N$ is by definition an odd dimensional Riemannian manifold whose cone $C(N)$ is a K\"ahler manifold. If the K\"ahler cone manifold
is toric then the Sasaki manifold is said to be toric. 
It is proven in \cite{FOW} and \cite{CFO} that a Sasaki-Einstein metric exists on a toric Sasaki manifold obtained from a toric diagram.
A typical example is when $N$ is the
standard sphere of real dimension $2m+1$. Then its cone is $\bfC^{m+1} - \{\bfo\}$. It is natural to expect that 
we can extend results on special Lagrangian submanifolds or self-similar solutions on $\mathbb C^{m+1}$ to
these toric Calabi-Yau cones of height $1$. In Theorem \ref{SLagtoric} we 
construct examples of complete special 
Lagrangian manifolds on toric Calabi-Yau cones
using the ideas of \cite{IonelMinoo08} and \cite{Kawai11}. 
This construction includes the examples given in  Theorem 3.1 in III.3 of Harvey-Lawson \cite{HarveyLawson82Acta}.
Further construction of examples of special Lagrangians and Lagrangian self-similar solutions are given in the third author's subsequent paper \cite{Yamamoto},
in which it is shown that, for any positive integer $g$, there are toric Calabi-Yau 3-dimensional cones including Lagrangian
self-shrinkers diffeomorphic to $\Sigma_g \times S^1$ where $\Sigma_g$ is a compact orientable surface of genus $g$.

In section 8 we also study the infinitesimal deformations of special Lagrangian cone $C(\Sigma) \subset C(N)$ over a Legendrian 
submanifold $\Sigma$ in a Sasaki-Einstein manifold $N$. We show that the parameter space $\mathcal{H}_{C(\Sigma)}$ 
of those infinitesimal deformations is isomorphic to
\begin{eqnarray}
{\rm Ker}(\Delta_\Sigma - 2n) = \{\varphi \in C^{\infty}(\Sigma);\ \Delta_\Sigma \varphi = 2n \varphi \},\nonumber
\end{eqnarray}
see Theorem \ref{infin.deform.sp}.
This is also proved by Lemma 3.1 of \cite{Ohnita}, although the proof in this paper is different from \cite{Ohnita}.

This paper is organized as follows. In section 2 we show fundamental formulas on mean curvature flows in 
Riemannian cone manifolds. In section 3 we show the finite time blowup of the mean curvature from
a compact manifold (Theorem \ref{blowup}).  Section 4 
is devoted to the proof of the monotonicity formula (Theorem \ref{mnf}). In section 5 we see that the type $\mathrm I$ singularity is preserved
under parabolic rescaling. In section 6 we see that we obtain a self-similar solution by parabolic rescaling at a type $\mathrm I_c$ singularity.
In section 7 we construct special Lagrangians in toric Calabi-Yau cones. In section 8 we study the infinitesimal deformations of
special Lagrangian cones in Calabi-Yau cones.


\section{Self-similar solutions to the mean curvature flows on Riemannian cone manifolds}

Let $F : M \to V$ be an immersion of  an $m$-dimensional manifold $M$ 
into an $m+k$-dimensional Riemannian manifold $(V, g)$. Thus the differential $F_{\ast x}:T_{x}M \to T_{F(x)}V$ is injective 
for every $x\in M$, and we have a natural orthogonal decomposition 
of the vector bundle 
$$F^{\ast}(TV) \cong TM \oplus T^{\bot}M$$
where $T^{\bot}M \to M$ 
is the normal bundle. 
Denote by $\bot$ (resp. $\top$) the projection
$\bot:F^{\ast}(TV)\rightarrow T^{\bot}M$ (resp. $\top:F^{\ast}(TV)\rightarrow TM$). The second 
fundamental form $\mathrm{II}$ of the immersion $F:M\rightarrow V$ is a section of the vector bundle
$T^{\bot}M\otimes(\otimes^{2}T^{*}M)$ defined by $\mathrm{II}(X,Y)=(\nabla_{F_{*}(X)}F_{*}(Y))^{\bot}$ 
for $X,Y\in\Gamma(TM)$. Here $\nabla$ is the Levi-Civita connection of $(V, g)$. 
The mean curvature vector field $H$ of $F:M\rightarrow V$ is a section of $T^{\bot}M$ defined by 
$H=\mathrm{tr}\ \mathrm{II}$, where the trace is taken with respect to the Riemannian metric 
$F^{*}(g)$ on $M$.

For the actual computations one often needs local expressions of the mean curvature vector. 
Let $x^1,\ \cdots, x^m$ and $y^1, \cdots, y^n$ be local coordinate charts around $p\in U\subset M$ and 
$F(p)\in U' \subset V$ such that $F\vert_{U}:U\rightarrow U'$ is an embedding. 
Write $F^\alpha(x^1, \cdots, x^m) = y^\alpha(F(x^1, \cdots, x^m))$. Then we have the induced metric
\begin{align}
g_{ij}=\frac{\partial F^{\alpha}}{\partial x^{i}}\frac{\partial F^{\beta}}{\partial x^{j}}g_{\alpha\beta}.\notag
\end{align}
where $g=g_{\alpha\beta}\,dy^{\alpha}\otimes dy^{\beta}$ is the Riemannian metric on $U' \subset V$.
Here we use the indices $i,\ j,\ k, \ldots$ to denote the coordinates on $M$ and 
$\alpha,\ \beta,\ \gamma, \ldots$ to denote the coordinates on $V$.
The coefficients $H^{\alpha}$ of the  mean curvature vector field
\begin{align}
H=H^{\alpha}\frac{\partial}{\partial y^{\alpha}}\notag
\end{align}
are given by the Gau\ss' formula 
\begin{align}\label{mcv}
H^{\alpha}=g^{ij}\biggl(\frac{\partial^{2}F^{\alpha}}{\partial x^{i}\partial{x^{j}}}
-\Gamma_{ij}^{k}\frac{\partial F^{\alpha}}{\partial x^{k}}+
\Gamma_{\beta\gamma}^{\alpha}\frac{\partial F^{\beta}}{\partial x^{i}}\frac{\partial F^{\gamma}}{\partial x^{j}}\biggr).
\end{align}

Next we consider a smooth family of immersions $F:M\times (a,b)\rightarrow V$. Namely,
for every time $t$ in $(a,b) \subset \bfR$, $F_{t}:M\rightarrow V$ given by $p\mapsto F(p,t)$ is an immersion. 
We denote by $g_{t}$ the Riemannian metric $F_{t}^{*}(g)$ over $M$. 
For a fixed time $t_{0}$ in $(a,b)$, the variation vector field $(\partial F/\partial t)(\cdot,t_{0})$, considered as a section of $F_{t_0}^\ast TV$,
 is decomposed as
$$(\partial F/\partial t)(\cdot,t_{0}) = v_{t_{0}}^{\bot} + v_{t_{0}}^{\top}$$
where $v_{t_{0}}^{\bot}(p)$ and  $v_{t_{0}}^{\top}$ are respectively the sections of $T^\bot M$ and $TM$.

We denote by 
$\nabla^{t}$, $\mathrm{div}_{t}$, $\mathrm{II}_{t}$ and $H_{t}$ respectively 
the Levi-Civita connection on $(M,g_{t})$, 
the divergence with respect to $g_{t}$, 
the second fundamental form and 
the mean curvature vector field of the immersion $F_{t}:M\rightarrow V$. 

Then following proposition is well-known as the ``first variation formula". 
\begin{proposition}\label{fvf}
For every $p$ in $M$, two tangent vectors $X,Y$ at $p$ and a compactly supported
integrable function $f$ on $M$, we have 
\begin{align}
& \frac{d}{dt}\biggl\vert_{t=t_{0}}g_{t}(X,Y)
=g_{t_{0}}(\nabla^{t_{0}}_{X}v_{t_{0}}^{\top},Y)
+g_{t_{0}}(X,\nabla^{t_{0}}_{Y}v_{t_{0}}^{\top})
-2g(\mathrm{II}_{t_{0}}(X,Y),v_{t_{0}}^{\bot}(p))\notag\\
& \frac{d}{dt}\biggl\vert_{t=t_{0}}\int_{M}fdv_{g_{t}}
=\int_{M}f(\mathrm{div}_{t_{0}}(v_{t_{0}}^{\top})
-g(H_{t_{0}},v_{t_{0}}^{\bot}))dv_{g_{t_{0}}}. \notag
\end{align}
\end{proposition}

Let 
$F:M\times [0,T)\rightarrow V$ be evolving by mean curvature flow with initial condition 
$F_{0}:M\rightarrow V$:
\begin{equation}\label{mcfc}
\begin{aligned}
& \frac{\partial F}{\partial t}(p,t)=H_{t}(p) &&~for~all~(p,t)\in M\times[0,T)\\
& F(p,0)=F_{0}(p) &&~for~all~p\in M.
\end{aligned}
\end{equation}

Applying the first variation formula in Proposition \ref{fvf} to the mean curvature flows, we obtain 
following well-known properties for mean curvature flows. 
\begin{proposition}
If $F:M\times [0,T)\rightarrow V$ is a mean curvature flow then the following equation holds.
\begin{align}
\frac{d}{dt}\biggl\vert_{t=t_{0}}\sqrt{\det(g_{t,ij})}=-\vert H_{t_{0}} \vert_{g}^2\sqrt{\det(g_{t_{0},ij})}. \label{kihonn}
\end{align}
If M is compact we also have
\begin{align}
\frac{d}{dt}\biggl\vert_{t=t_{0}}\mathrm{Vol}_{g_{t}}(M)=-\int_{M}\vert H_{t_{0}}\vert_{g}^2\,dv_{g_{t_{0}}}.\notag
\end{align}
\end{proposition}
\begin{proof}
Because we consider the mean curvature flow, $v_{t_{0}}=H_{t_{0}}$ and therefore
$$v_{t_{0}}^{\top}=0 \qquad \mathrm{and}\qquad v_{t_{0}}^{\bot}(p)=H_{t_{0}}(p).$$ 
It then follows from Proposition \ref{fvf} that 
\begin{align}
& \frac{d}{dt}\biggl\vert_{t=t_{0}}g_{t,ij}
=-2g(\mathrm{II}_{t_{0},ij},H_{t_{0}}).\notag
\end{align}
Then the first formula (\ref{kihonn}) follows from the well-known formula for the derivative of the determinant.
To prove second formula, simply let $f\equiv 1$ on $M$ in the first variation formula. 
\end{proof}


Recall that, 
for an $n$-dimensional Riemannian manifold $(N,g)$, we define the Riemannian cone manifold 
$(C(N),\overline{g})$ over $(N,g)$ by $C(N)=N\times\mathbb{R}^{+}$ and $\overline{g}=dr^2+r^2g$ where $r$ is the 
standard coordinate of $\mathbb{R}^{+}$. Note that $C(S)$ does not contain the apex.

The most typical example of a cone is the case when $N$ is the standard sphere $S^{n}$ in $\bfR^{n+1}$. In this case the cone is
$\bfR^{n+1} - \{\bfo\}$. For a map $F:M\rightarrow \bfR^{n+1}$, one can consider the position vector of $F(p)$ for $p \in M$, and using it, one can
define self-similar solutions
$$ H = \lambda F^{\bot}$$ 
where $\lambda$ is a constant.

We can extend this idea to maps into Riemannian cone manifolds.
Namely, 
for a smooth map $F:M\rightarrow C(N)$ and $p$ in $M$, 
we define the position vector $\ArrowF$ of $F$ at $p \in M$ by
\begin{align}
\ArrowF(p)= r(F(p))\frac{\partial}{\partial r}~~\in T_{F(p)}C(N). \notag
\end{align}
With respect to the bundle decomposition of 
$$ F^\ast T_{F(p)}C(N)  \cong T_{p}M \oplus T^{\bot}_{p}M,$$
we decompose $\ArrowF(p)$ as
$$ \ArrowF(p) = \ArrowF^{\top}(p) + \ArrowF^{\bot}(p).$$
Then we can
define self-similar solutions by
$$ H = \lambda \ArrowF^{\bot}.$$

For a Riemannian cone manifold $(C(N),\overline{g})$ over an $n$-dimensional Riemannian manifold $(N,g)$ and a point $q$ in $C(N)$, 
local coordinates $(y^{\alpha})_{\alpha=1}^{n+1}$ around $q$ are said to be
{\it associated with normal local coordinates} of $N$ when the part of coordinate $(y^{\alpha})_{\alpha=1}^{n}$ becomes 
normal local coordinates of $(N,g)$ around $\pi_N(q)$ and $y^{n+1}$ is the standard coordinate of $\mathbb{R}^{+}$, that is, $y^{n+1}=r$.  
Here, $\pi_N$ is the projection of the cone manifold $C(N) \cong N \times \bfR^+$ onto the first factor $N$.

Note that under local coordinates associated with normal local coordinates of $N$, we have $r\circ F= r(F) = F^{n+1}$ for a given map $F:M\rightarrow C(N)$. 

Let $(x^{i})_{i=1}^{m}$ be normal local coordinates centered at $p$ of the Riemannian manifold $(M,F^{*}(\overline{g}))$, 
and $(y^{\alpha})_{\alpha=1}^{n+1}$ local coordinates of $(C(N),\overline{g})$ associated with normal local coordinates centered at $\pi_N(F(p))$ 
of $(N,g)$. 
Then calculating only $(n+1)$-th coefficient $H^{n+1}(p)$ of mean curvature vector at $p$, namely, the coefficient of $\partial/\partial y^{n+1}(=\partial/\partial r)$, 
for the local expression of the mean curvature vector (\ref{mcv}), we obtain the following local expression for $H^{n+1}(p)$; 
\begin{align}
H^{n+1}(p)=\sum_{i=1}^{m}\frac{\partial^2 r(F)}{{\partial x^{i}}^2}(p) - r(F(p))\sum_{i=1}^{m}\sum_{\alpha=1}^{n}\biggl(\frac{\partial F^{\alpha}}{\partial x^{i}}(p)\biggr)^2.\label{h}
\end{align}
This easily follows from 
$$ \Gamma^{n+1}_{\alpha\beta} = - r g_{\alpha\beta}$$
for $1 \le \alpha,\ \beta \le n$.
\section{Finite time singularity for mean curvature flows}

If the ambient space is the Euclidean space $\mathbb{R}^{m+k}$ and an initial date manifold $M$ is compact, then the mean curvature flow does not have
a long time solution. It is a well-known result of Huisken:
\begin{theorem}[Huisken \cite{Huisken}]\label{Huisken}
Let $F_0:M\rightarrow \mathbb{R}^{m+k}$ be an immersion of a compact $m$-dimensional manifold $M$. 
Then the maximal time $T$ of existence of a solution $F:M\times[0,T)\rightarrow \mathbb{R}^{m+k}$
of the mean curvature flow with initial immersion $F_0$ is finite.
\end{theorem}
The proof follows by applying the parabolic maximum principle to the function
$f=|F|^2+2mt$ which satisfies the  evolution equation $\frac{d}{dt}f=\Delta f$. One can show 
$T\leq\frac{1}{2m}\max|F_0|^2$, from which Theorem \ref{Huisken} follows. Using the position vector in a cone as defined in (\ref{pv}), 
we can extend this result when 
the ambient space is a Riemannian cone manifold as follows.
\begin{theorem}\label{blowup}
Let $(C(N),\overline{g})$ be the Riemannian cone manifold over a Riemannian manifold $(N,g)$ of 
dimension $n$, $M$ a
 compact manifold of dimension $m$ and 
$F:M\times[0,T)\rightarrow C(N)$ a mean curvature flow with initial condition $F_{0}:M\rightarrow C(N)$. 
Then the maximal time $T$ of existence of the mean curvature flow is finite. 
\end{theorem}

Before the proof of this theorem, we want to prepare some lemmas. 
\begin{lemma}\label{lem1}
Let $(C(N),\overline{g})$ be a Riemannian cone manifold over a Riemannian manifold $(N,g)$ of dimension $n$ and 
$F:M\rightarrow C(N)$ an immersion of a manifold $M$ of dimension $m$. Then the following equation holds.
\begin{align}
\Delta (r^2(F)) =2(\overline{g}(H,\ArrowF)+m),\notag
\end{align}
where $\Delta$ is the Laplacian on $(M,F^{*}(\overline{g}))$.
\end{lemma}
\begin{proof}
Fix a point $p$ in $M$. We take normal local coordinates $(x^{i})_{i=1}^{m}$ of $(M,F^{*}(\overline{g}))$ centered at $p$ and 
local coordinates $(y^{\alpha})_{\alpha=1}^{n+1}$ of $(C(S),\overline{g})$ associated with normal local coordinates of $(N,g)$ centered at $\pi_N(F(p))$. 
Note that under these coordinates, $y^{n+1}=r$ and $F^{n+1}=r \circ F = r(F)$. 
First of all, by the local expression of $H^{n+1}(p)$ in (\ref{h}), we have the following equalities; 
\begin{align}
\overline{g}(H(p),\ArrowF(p))
=&H^{n+1}(p)r(F(p)) \notag\\
=&r(F(p))\sum_{i=1}^{m}\frac{\partial^2 r(F)}{{\partial x^{i}}^2}(p)-r(F(p))^2\sum_{i=1}^{m}\sum_{\alpha=1}^{n}\biggl(\frac{\partial F^{\alpha}}{\partial x^{i}}(p)\biggr)^2.\label{a}
\end{align}
Since $(F^{*}\overline{g})(\partial/\partial x^{i},\partial/\partial x^{i})=1$ at $p$, we have 
\begin{align}
m
=&\sum_{i=1}^{m}(F^{*}\overline{g})\biggl(\frac{\partial}{\partial x^{i}}(p),\frac{\partial}{\partial x^{i}}(p)\biggr)\notag\\
=&r(F(p))^2 \sum_{i=1}^{m}\sum_{\alpha=1}^{n}\biggl(\frac{\partial F^{\alpha}}{\partial x^{i}}(p)\biggr)^2+\sum_{i=1}^{m}\biggl(\frac{\partial r(F)}{{\partial x^{i}}}(p)\biggr)^2.\label{b}
\end{align}
Adding above two equations (\ref{a}) and (\ref{b}), we have

\begin{equation}\label{c}
\overline{g}(H(p),\ArrowF(p))+m
= r(F(p)) \sum_{i=1}^{m}\frac{\partial^2 r(F)}{{\partial x^{i}}^2}(p)+\sum_{i=1}^{m}\biggl(\frac{\partial r(F)}{{\partial x^{i}}}(p)\biggr)^2. 
\end{equation}

Since we took $(x^{i})_{i=1}^{m}$ as normal local coordinates of $(M,F^{*}(\overline{g}))$ centered at $p$, 
the Laplacian $\Delta$ is $\sum_{i=1}^{m}(\partial/\partial x^{i})^2$, 
and thus we have at $p$ 
\begin{align}
\Delta r^2(F)
=&\sum_{i=1}^{m}\frac{\partial^2 r^2(F)}{\partial {x^{i}}^2}\notag\\
=&2\Biggl(r(F)\sum_{i=1}^{m}\frac{\partial^2 r(F)}{{\partial x^{i}}^2}+\sum_{i=1}^{m}\biggl(\frac{\partial r(F)}{{\partial x^{i}}}\biggr)^2\Biggr). \label{d}
\end{align}
Thus from (\ref{c}) and (\ref{d}) we have shown that $\Delta r^2(F) = 2(\overline{g}(H,\ArrowF)+m)$. 
\end{proof}
\begin{lemma}\label{lem2}
Let $(C(N),\overline{g})$ be a Riemannian cone manifold over an $n$-dimensional Riemannian manifold $(N,g)$, 
$M$ an $m$-dimensional manifold and $F:M\times[0,T)\rightarrow C(N)$ be a mean curvature flow with initial condition $F_{0}:M\rightarrow C(N)$.
Then for any fixed time $t$ in $[0,T)$ the following equality holds; 
\begin{align}
2\overline{g}(H_{t},\ArrowF_{t})=\frac{\partial}{\partial t}{r^2(F_t)}.\label{bibunn}
\end{align}
\end{lemma}
\begin{proof}
Fix a point $p$ in $M$. Take local coordinates $(y^{\alpha})_{\alpha=1}^{n+1}$ of $C(S)$ associated with normal local coordinates of $N$. 
Note that under these coordinates, $y^{n+1}=r$ and $F_{t}^{n+1}=r(F_{t})$.
Since $F$ satisfies the mean curvature flow condition (\ref{mcfc}), the following equalities hold;  
\begin{eqnarray*}
\overline{g}(H_{t}(p),\ArrowF_{t}(p))
&=&\overline{g}\biggl(\frac{\partial F}{\partial t}(p,t),\ArrowF_{t}(p)\biggr)\\
&=& r(F_{t}(p))\frac{\partial}{\partial t}r(F_{t }(p)) = \frac{1}{2}\frac{\partial}{\partial t}r^2(F_{t }(p)),
\end{eqnarray*}
from which (\ref{bibunn}) follows.
\end{proof}
\noindent
Now we are in a position to prove Theorem \ref{blowup}.
\begin{proof}[Proof of Theorem \ref{blowup}.]
Let $f:M\times[0,T)\rightarrow \mathbb{R}$ be a function defined by
\begin{align}
f(p,t)=r^2(F_{t}(p)) + 2 m t\notag.
\end{align}
For a fixed time $t$ in $[0,T)$, by Lemma \ref{lem1} and Lemma \ref{lem2}, 
\begin{eqnarray*}
\frac{\partial f}{\partial t}
&=&2\overline{g}(H_{t},\ArrowF_{t})+2 m \\
&=&\Delta_{t}r^2(F_{t}) =\Delta_{t}f(\cdot,t)
\end{eqnarray*}
where $\Delta_{t}$ is the Laplacian with respect to the metric ${F_{t}}^{*}(\overline{g})$ on $M$. 
Since $M$ is compact, there is a maximum of $f(\cdot,0)(=r^2(F_{0}))$ on $M$, which 
we denote by $C_{0}$. 
By applying the maximum principle to the function $f$, it follows that $f(p,t)=r^2(F_{t }(p))+2 m t \leq C_{0}$ on $M\times [0,T)$. 
Therefore we obtain the following inequalities;
\begin{align}
t\leq\frac{C_{0}-r^2(F_{t}(p))}{2m}\leq\frac{C_{0}}{2m}\notag
\end{align}
for all $t$ in $[0,T)$. This means that the maximal time $T$ is finite. 
\end{proof}
\section{Monotonicity formula}

Next we turn to the monotonicity formula. 
For a fixed time $T$ in $\mathbb{R}$, we define the backward heat kernel 
$\rho_{T}:\mathbb{R}\times(-\infty,T)\rightarrow\mathbb{R}$ as follows; 
\begin{align}
\rho_{T}(y,t)=\frac{1}{(4\pi(T-t))^{m/2}}\exp\biggl(-\frac{y^2}{4(T-t)}\biggr).\notag
\end{align}
To simplify the notations, we use following abbreviation;
\begin{align}
&\int_{M_{t}}\rho_{T}:=\int_{M}\rho_{T}(r(F_{t}(p)),t)dv_{g_{t}}\notag\\
&\int_{M_{t}}\rho_{T}\biggl\vert\frac{\ArrowF^{\bot}}{2(T-t)}+H\biggr\vert_{\overline{g}}^2\notag
:=\int_{M}\rho_{T}(r(F_{t}(p)),t)\biggl\vert\frac{\ArrowF_{t}^{\bot}(p)}{2(T-t)}+H_{t}(p)\biggr\vert_{\overline{g}}^2dv_{g_{t}}.\notag
\end{align}
Then Huisken's monotonicity formula for a cone is the following.
\begin{theorem}[Monotonicity formula]\label{mnf}
Let $M$ be a compact $m$-dimensional manifold without boundary, 
$(C(N),\overline{g})$ the Riemannian cone manifold over an $n$-dimensional Riemannian manifold $(N,g)$ and 
$F:M\times [0,T)\rightarrow C(N)$ the mean curvature flow with initial condition 
$F_{0}:M\rightarrow C(N)$. 
Then the following equation holds;
\begin{align}
\frac{d}{dt}\int_{M_{t}}\rho_{T}=-\int_{M_{t}}\rho_{T}\biggl\vert\frac{\ArrowF^{\bot}}{2(T-t)}+H\biggr\vert_{\overline{g}}^2. \label{mono1}
\end{align}
\end{theorem}

\begin{proof}
First we calculate the left term of (\ref{mono1}) using (\ref{kihonn}). 
\begin{align}
\frac{d}{dt}& \int_{M}\rho_{T}(r(F_{t}(p)),t)dv_{g_{t}}\notag\\
=&\frac{d}{dt} \int_{M}\frac{1}{(4\pi(T-t))^{m/2}}\exp\biggl(-\frac{r^2(F_{t}(p))}{4(T-t)}\biggr)\sqrt{\det(g_{t,ij})}\,dx^{1}\wedge\dots\wedge dx^{m}\notag\\
=&\int_{M}\rho_{T}(r(F_{t}(p)),t)\biggl(\frac{m}{2(T-t)}-\frac{r^2(F_{t}(p))}{4(T-t)^2}\notag\\
&\hspace{40mm}-\frac{r(F_{t}(p))\bigl(\frac{\partial}{\partial t}r(F_{t}(p))\bigr)}{2(T-t_{0})}-\vert H_{t}(p) \vert_{\overline{g}}^2\biggr)dv_{g_{t}}.\label{f}
\end{align}
It is clear that
\begin{eqnarray}
\vert\ArrowF_{t}(p)\vert_{\overline{g}}^{2}
&=&\overline{g}\biggl(r(F_{t}(p))\frac{\partial}{\partial r},r(F_{t}(p))\frac{\partial}{\partial r}\biggr)\nonumber\\
&=&r^2(F_{t}(p)). \label{normofpoint}
\end{eqnarray}
Substituting (\ref{bibunn}) and (\ref{normofpoint}) in (\ref{f}), we have following formula;
\begin{align}
&\frac{d}{dt}\int_{M}\rho_{T}(r(F_{t}(p)),t)dv_{g_{t}}\notag\\
=&\int_{M}\rho_{T}(r(F_{t}(p)),t)\Biggl(\frac{m}{2(T-t)}-\frac{\vert\ArrowF_{t}(p)\vert_{\overline{g}}^{2}}{4(T-t)^2}\notag\\
&\hspace{50mm}-\frac{\overline{g}(H_{t}(p),\ArrowF_{t}(p))}{2(T-t)}-\vert H_{t}(p)\vert_{\overline{g}}^2\biggr)dv_{g_{t}}.\label{g}
\end{align}

Let $t$ and $p$ be fixed. We take normal local coordinates $(x^{i})_{i=1}^{m}$ centered at $p$ with respect to the
Riemannian metric $g_{t}(=F_{t}^*(\overline{g}))$ and local coordinates $(y^{\alpha})_{\alpha=1}^{n+1}$ around $F_{t}(p)$ associated with normal local coordinates of $(N,g)$. 
Under these coordinates, the Laplacian $\Delta_{t}$ with respect to $g_{t}$ is $\partial^2/\partial {x^{1}}^{2}+\dots+\partial^2/\partial {x^{m}}^{2}$ at $p$. 
Under these coordinates we have following equations at the fixed $t$ and $p$;
\begin{eqnarray}
\Delta_{t}\rho_{T}(r(F_{t}),t)\
&=& \sum_{i=1}^{m}\frac{\partial^2}{\partial {x^{i}}^2}\biggl\vert_{x=p} \rho_{T}(r(F_{t}),t)\notag\\
&=&\sum_{i=1}^{m}\frac{\partial}{\partial x^{i}}\biggl\vert_{x=p} \biggl(\frac{\partial}{\partial x^{i}}\rho_{T}(r(F_{t}),t)\biggr)\notag\\
&=&\sum_{i=1}^{m}\frac{\partial}{\partial x^{i}}\biggl\vert_{x=p} \Biggl(\rho_{T}(r(F_{t}),t)\biggl(-\frac{r(F_{t})\bigl(\frac{\partial}{\partial x^{i}}r(F_{t})\bigr)}{2(T-t)}\biggr)\Biggr)\notag\\
&=&\rho_{T}(r(F_{t}),t)\Biggl(\frac{r^2(F_{t}) \bigl(\frac{\partial}{\partial x^{i}}r(F_{t})\bigr)^2}{4(T-t)^2}\notag\\
&&\hspace{10mm}-\frac{\bigl(\frac{\partial}{\partial x^{i}}r(F_{t})\bigr)^2}{2(T-t)}-\frac{r(F_{t})\bigl(\frac{\partial^2}{\partial {x^{i}}^2}r(F_{t})\bigr)}{2(T-t)}\Biggr). \label{Laplacian}
\end{eqnarray}
Furthermore we want to express $\ArrowF_{t}^{\top}(p)$ under these coordinates. Now by our choice of the local coordinates of $(x^{i})_{i=1}^{m}$, 
it is clear that 
\begin{align}
\overline{g}\Biggl(F_{t*}(p)\biggl(\frac{\partial}{\partial x^{i}}\biggr),F_{t*}(p)\biggl(\frac{\partial}{\partial x^{j}}\biggr)\Biggr)=\delta_{ij}.\label{delta}
\end{align}
Note that $y^{n+1}=r$ and $F_{t}^{n+1}=r(F_{t})$. The following equalities hold;
\begin{eqnarray}
\ArrowF_{t}^{\top}(p)
&=&\sum_{i=1}^{m}\overline{g}\Biggl(\ArrowF_{t}^{\top}(p),F_{t*}(p)\biggl(\frac{\partial}{\partial x^{i}}\biggr)\Biggr)F_{t*}(p)
\biggl(\frac{\partial}{\partial x^{i}}\biggr)\notag\\
&=&\sum_{i=1}^{m}\overline{g}\Biggl(\ArrowF_{t}(p),F_{t*}(p)\biggl(\frac{\partial}{\partial x^{i}}\biggr)\Biggr)F_{t*}(p)
\biggl(\frac{\partial}{\partial x^{i}}\biggr)\notag\\
&=&\sum_{i=1}^{m}\overline{g}\Biggl(r(F_{t}(p))\frac{\partial}{\partial r},\sum_{\alpha=1}^{n+1}\frac{\partial F_{t}^{\alpha}(p)}{\partial x^{i}}
\frac{\partial}{\partial y^{\alpha}}\Biggr)F_{t*}(p)\biggl(\frac{\partial}{\partial x^{i}}\biggr)\notag\\
&=&r(F_{t}(p))\sum_{i=1}^{m}\frac{\partial r(F_{t}(p))}{\partial x^{i}}F_{t*}(p)\biggl(\frac{\partial}{\partial x^{i}}\biggr).\label{top}
\end{eqnarray}
Using (\ref{delta}) and (\ref{top}), we can express the norm of $\ArrowF_{t}^{\top}(p)$ as follows;
\begin{eqnarray}
\vert\ArrowF_{t}^{\top}(p)\vert_{\overline{g}}^{2}
&=&\overline{g}\biggl(\ArrowF_{t}^{\top}(p),\ArrowF_{t}^{\top}(p)\biggr)\notag\\
&=&r^2(F_{t}(p))\sum_{i=1}^{m}\biggl(\frac{\partial r(F_{t})}{\partial x^{i}}(p)\biggr)^2 .\label{normoftop}
\end{eqnarray}
Applying (\ref{c}) for $F_{t}$ and using (\ref{Laplacian}) and (\ref{normoftop}), we have the following equality;
\begin{align}
&\Delta_{t}\rho_{T}(r(F_{t}(p)),t)\notag\\
=&\rho_{T}(r(F_{t}(p)),t)\Biggl(\frac{\vert\ArrowF_{t}^{\top}(p)\vert_{\overline{g}}^{2}}{4(T-t)^2}-\frac{m}{2(T-t)}-\frac{\overline{g}(H_{t}(p),\ArrowF_{t}(p))}{2(T-t)}\biggr).\label{important}
\end{align}
In this equation (\ref{important}) there are no local coordinates $x^{i}$, so we have proven this equation (\ref{important}) for all $p$ in $M$ globally. 
The equation (\ref{important}) is equivalent to 
\begin{eqnarray}
\rho_{T}(F_{tR}(p),t)\frac{m}{2(T-t)}
&=&-\Delta_{t}\rho_{T}(F_{tR}(p),t)\notag\\
&+& \rho_{T}(F_{tR}(p),t)\Biggl(\frac{\vert\overline{F_{t}}^{\top}(p)\vert_{\overline{g}}^{2}}{4(T-t)^2}-\frac{\overline{g}(H_{t}(p),\overline{F_{t}}(p))}{2(T-t)}\biggr).\label{i}
\end{eqnarray}
Substituting (\ref{i}) in (\ref{g}), we have following equalities;
\begin{align}
&\frac{d}{dt}\int_{M}\rho_{T}(r(F_{t}(p)),t)dv_{g_{t}}\notag\\
=&-\int_{M}\Delta_{t}\rho_{T}(r(F_{t}(p)),t)dv_{g_{t}}\notag\\
&+\int_{M}\rho_{T}(r(F_{t}(p),t))\Biggl(\frac{\vert\ArrowF_{t}^{\top}(p)\vert_{\overline{g}}^{2}}{4(T-t)^2}-\frac{\vert\ArrowF_{t}(p)\vert_{\overline{g}}^{2}}{4(T-t)^2}\notag\\
&\hspace{40mm}-2\times\frac{\overline{g}(H_{t}(p),\ArrowF_{t}(p))}{2(T-t)}-\vert H_{t_{0}}(p)\vert_{\overline{g}}^2\biggr)dv_{g_{t}}\notag\\
=&\int_{M}\rho_{T}(r(F_{t}(p)),t)\Biggl(-\frac{\vert\ArrowF_{t}^{\bot}(p)\vert_{\overline{g}}^{2}}{4(T-t)^2}\notag\\
&\hspace{40mm}-2\times\frac{\overline{g}(H_{t}(p),\ArrowF_{t}^{\bot}(p))}{2(T-t)}-\vert H_{t}(p)\vert_{\overline{g}}^2\biggr)dv_{g_{t}}\notag\\
=&-\int_{M}\rho_{T}(r(F_{t}(p)),t)\Biggl\vert\frac{\ArrowF_{t}^{\bot}(p)}{2(T-t)}+H_{t}(p)\Biggr\vert_{\overline{g}}^2dv_{g_{t}}.
\end{align}
This completes the proof of Theorem \ref{mnf}. 
\end{proof}
\section{Singularities and the parabolic rescaling}

In this section we see that the property that a mean curvature flow develops type I singularities is preserved under parabolic rescaling.

\begin{proposition}\label{aa}
Let $M$ be an m-dimensional manifold and $(C(N),\overline{g})$ the Riemannian cone manifold over an n-dimensional Riemannian manifold $(N,g)$. 
If a map $F:M\times[0,T)\rightarrow C(N)$ is a mean curvature flow, then the parabolic rescaling of $F$ of scale $\lambda$ is 
also the mean curvature flow. 
\end{proposition}
\begin{proof}
Fix $(p_{0},s_{0})$ in $M\times[-\lambda^2T,0)$. 
Let $t=T+s/\lambda^2$ and $t_{0}=T+s_{0}/\lambda^2$. 
Let $(x^{i})_{i=1}^{m}$ be local coordinates of $M$ around $p_{0}$. 
Let $(y^{\alpha})_{\alpha=1}^{n+1}$ be local coordinates of $C(N)$ around $F^{\lambda}(p_{0},s_{0})$ associated with local coordinates $N$. 
Put 
$$
{g^{\lambda}_{s_{0}}}_{ij}=(F^{\lambda *}_{s_{0}}\overline{g})\biggl(\frac{\partial}{\partial x^{i}},\frac{\partial}{\partial x^{j}}\biggr)\ \ 
\mathrm{and}\ \ {g_{t_{0}}}_{ij}=(F_{t_{0}}^{*}\overline{g})\biggl(\frac{\partial}{\partial x^{i}},\frac{\partial}{\partial x^{j}}\biggr).
$$
Then one can easily show that 
\begin{equation}\label{glambda}
{g^{\lambda}_{s_{0}}}_{ij}=\lambda^2{g_{t_{0}}}_{ij}.
\end{equation}
Thus the Christoffel symbols  ${\Gamma^{\lambda}_{s_{0}}}^{i}_{jk}$ with respect to $g^{\lambda}_{s_{0}}$ and 
${\Gamma_{t_{0}}}^{i}_{jk}$ with respect to $g_{t_{0}}$ are related by 
$${\Gamma^{\lambda}_{s_{0}}}^{i}_{jk}={\Gamma_{t_{0}}}^{i}_{jk}.$$ 
One can also compute the Christoffel symbols of the Riemannian cone manifold $C(N)$ as follows. If $1\leq\alpha,\beta,\gamma\leq n$, then 
$\overline{\Gamma}^{\alpha}_{\beta\gamma}(F^{\lambda}_{s_{0}}(p_{0}))=\overline{\Gamma}^{\alpha}_{\beta\gamma}(F_{t_{0}}(p_{0})).$
If $1\leq\beta,\gamma\leq n$ and $\alpha=n+1$ then 
$\overline{\Gamma}^{n+1}_{\beta\gamma}(F^{\lambda}_{s_{0}}(p_{0}))=\lambda\overline{\Gamma}^{n+1}_{\beta\gamma}(F_{t_{0}}(p_{0})),$
and 
if $1\leq\alpha,\gamma\leq n$ and $\beta=n+1$ then 
$\overline{\Gamma}^{\alpha}_{n+1 \gamma}(F^{\lambda}_{s_{0}}(p_{0}))=\frac{1}{\lambda}\overline{\Gamma}^{\alpha}_{n+1 \gamma}(F_{t_{0}}(p_{0}))
$.
By using these and the formula (\ref{mcv}), one can show that 
the mean curvature vectors $H_{t_{0}}$ of $F_{t_{0}}$ and 
$H^{\lambda}_{s_{0}}$ of $F^{\lambda}_{s_{0}}$ are related by 
\begin{equation}\label{aaa1}
H^{\lambda\,\alpha}_{s_{0}}(p_{0})
=\frac{1}{\lambda^2}H_{t_{0}}^{\alpha}(p_{0}).
\end{equation}
for $1 \le \alpha \le n$ and 
\begin{equation}
H^{\lambda\,n+1}_{s_{0}}(p_{0})
=\frac{1}{\lambda}H_{t_{0}}^{n+1}(p_{0}).\label{aaa2}
\end{equation}
Now suppose that $F$ is a mean curvature flow, so $F$ satisfies 
\begin{align}
F_{*}(p_{0},t_{0})\biggl(\frac{\partial}{\partial t}\biggr)=H_{t_{0}}(p_{0}).\notag
\end{align}
Then 
\begin{eqnarray}
{F^{\lambda}}_{*}(p_{0},s_{0})\biggl(\frac{\partial}{\partial s}\biggr)
&=&\frac{1}{\lambda^2}\sum_{\alpha=1}^{n}H_{t_{0}}^{\alpha}(p_{0})\frac{\partial}{\partial y^{\alpha}}(p_{0})
+\frac{1}{\lambda}H_{t_{0}}^{n+1}(p_{0})\frac{\partial}{\partial y^{n+1}}(p_{0})\notag\\
&=&\sum_{\alpha=1}^{n}H^{\lambda\,\alpha}_{s_{0}}(p_{0})\frac{\partial}{\partial y^{\alpha}}(p_{0})
+H^{\lambda\,n+1}_{s_{0}}(p_{0})\frac{\partial}{\partial y^{n+1}}(p_{0})
=H^{\lambda}_{s_{0}}(p_{0}).\notag
\end{eqnarray}
This means that $F^{\lambda}$ is the mean curvature flow. This completes the proof of Proposition \ref{aa}.
\end{proof}
\begin{proposition}\label{prop5}
Let $M$ be an $m$-dimensional manifold and $C(N)$ the Riemannian cone over an $n$-dimensional Riemannian manifold $(N,g)$. 
Let $F : M \times[0,T)\rightarrow C(N)$ be a mean curvature flow. 
Then parabolic rescaling preserves the value of $\int_{M_{t}}\rho_{T}$. 
This means that for all $t$ in $(0,T)$ the following equation holds. 
\begin{align}
\int_{M_{t}}\rho_{T}=\int_{M^{\lambda}_{s}}\rho_{0}\notag
\end{align}
where $s=\lambda^2(t-T)$. Here 
we have used abbreviation for $\int_{M_{t}}\rho_{T}$ and $\int_{M^{\lambda}_{s}}\rho_{0}$ by 
\begin{align}
&\int_{M_{t}}\rho_{T}=\int_{M}\rho_{T}(r(F_{t}(p)),t)dv_{g_{t}}\notag\\
&\int_{M^{\lambda}_{s}}\rho_{0}=\int_{M}\rho_{0}(r(F^{\lambda}_{s}(p)),s)dv_{g^{\lambda}_{t}}.\notag
\end{align}
\end{proposition}
\begin{proof}
From the equation (\ref{glambda}) in the proof of the proposition \ref{aa}, we get
$$ 
\sqrt{\det({g^{\lambda}_{s}}_{ij})}=\lambda^m \sqrt{\det({g_{t}}_{ij})}\ \ \ \mathrm{and}\ \ \ 
dv_{g^{\lambda}_{s}}=\lambda^m dv_{g_{t}}.
$$
It follows that 
\begin{eqnarray}
\int_{M^{\lambda}_{s}}\rho_{0}
&=& \int_{M}\frac{1}{(4\pi(-s))^{m/2}}\exp\biggl(-\frac{r^2(F^{\lambda}_{s}(p))}{4(0-s)}\biggr)dv_{g^{\lambda}_{s}}\notag\\
&=&\int_{M}\frac{1}{(4\pi(\lambda^2(T-t)))^{m/2}}\exp\biggl(-\frac{\lambda^2 r^2(F_{t}(p))}{4\lambda^2(T-t)}\biggr)\lambda^m dv_{g_{t}}\notag\\
&=&\int_{M}\frac{1}{(4\pi(T-t))^{m/2}}\exp\biggl(-\frac{F_{tR}(p)^2}{4(T-t)}\biggr)dv_{g_{t}}
=\int_{M_{t}}\rho_{T}\notag
\end{eqnarray}
\end{proof}

\begin{proposition}\label{prop6}
Let $M$ be an $m$-dimensional manifold and $C(N)$ the Riemannian cone over an $n$-dimensional Riemannian manifold $(N,g)$. 
Let $F\times[0,T)\rightarrow C(N)$ be a mean curvature flow. 
Then the parabolic rescaling preserves the condition that the mean curvature flow develops a Type I singularity. 
\end{proposition}
\begin{proof}
We have only to show that following two statements are equivalent.

$\bullet$ \ There exists some $c>0$ such that $\sup _{M}\vert\mathrm{II}_{t}\vert^2\leq\frac{c}{T-t}~\mathrm{for~all~}t\in[0,T)$.

$\bullet$ \ There exists some $c'>0$ such that $\sup _{M}\vert\mathrm{II}^{\lambda}_{s}\vert^2\leq\frac{c'}{-s}~\mathrm{for~all~}s\in[-\lambda^2T,0)$.

\noindent
Here $\mathrm{II}_{t}$ and $\mathrm{II}^{\lambda}_{s}$ are the second fundamental form with respect to the immersion $F_{t}:M\rightarrow C(N)$ and $F^{\lambda}_{s}:M\rightarrow C(N)$ respectively. 

We can find a local expression of 
$\mathrm{II}^{\lambda\,\alpha}_{s\,ij}$ and $\mathrm{II}^{\alpha}_{t\,ij}$ immediately 
by removing the inverse of Riemannian metric tensors $g^{\lambda\,ij}_{s}$
($=\frac{1}{\lambda^2}g^{ij}_{t}$) from equalities (\ref{aaa1}) and (\ref{aaa2}).
Hence, we find that $\mathrm{II}^{\lambda\,\alpha}_{s\,ij}=\mathrm{II}^{\alpha}_{t\,ij}$ if $1\leq\alpha\leq n$, 
and $\mathrm{II}^{\lambda\,n+1}_{s\,ij}=\lambda\mathrm{II}^{n+1}_{t\,ij}$ if $\alpha=n+1$, where $s=\lambda^2(t-T)$. 
It then follows that 
\begin{align}
\vert \mathrm{II}^{\lambda}_{s} \vert^2(p)=\frac{1}{\lambda^2}\vert \mathrm{II}_{t} \vert^2(p).\label{same2}
\end{align}
Hence we get
\begin{align}
(T-t)\vert \mathrm{II}_{t} \vert^2=\frac{-s}{\lambda^2}\times\lambda^2 \vert \mathrm{II}^{\lambda}_{s} \vert^2=-s\vert \mathrm{II}^{\lambda}_{s} \vert^2.\label{same}
\end{align}
This mean that parabolic rescaling preserves the condition developing type I singularity. This completes the proof of Proposition \ref{prop6}.
\end{proof}
\section{Self-similar solutions}

This section is devoted to the proof of Theorem \ref{main}.

\begin{proof}[Proof of Theorem \ref{main}]
Take any increasing sequence $\{\lambda_{i}\}_{i=1}^{\infty}$ of the scales of the parabolic rescaling 
such that $\lambda_{i}\rightarrow\infty$ as $i\rightarrow\infty$. 
Let $F^{\lambda_{i}}:M\times[-\lambda_{i}^2T,0)\rightarrow C(N)$ be the parabolic rescaling of the  mean curvature flow $F : M \times[0,T)\rightarrow C(N)$. 
By Proposition \ref{aa}, $F^{\lambda_{i}}$ remains to be a mean curvature flow.

Since $F$ develops type $\mathrm{I}_c$ singularity and in particular type $\mathrm{I}$ singularity,
there exists a positive real number $C > 0$ suth that 
\begin{align}
\sup _{M}\vert\mathrm{II}_{t}\vert^2\leq\frac{C}{T-t}\notag
\end{align}
for all $t$ in $[0,T)$, and 
by Proposition \ref{prop6} the rescaled $F^{\lambda_{i}}$ also develops type $\mathrm{I}$ singularity
satisfying 
\begin{align}
\sup _{M}\vert\mathrm{II}^{\lambda_{i}}_{s}\vert^2\leq\frac{C}{-s}\notag
\end{align}
for all $s$ in $[-\lambda_{i}^2,0)$
with the same constant $C>0$ by (\ref{same}). 
When $s$ is restricted to the interval $[a,b]$, we have the following bound
\begin{align}
\vert\mathrm{II}^{\lambda_{i}}_{s}\vert^2\leq-\frac{C}{b}. \label{SSbound}
\end{align}
Hence we have a uniform bound of the second fundamental form, and since $F^{\lambda_i}$ satisfies the mean curvature flow, all 
the higher derivatives of the second fundamental form are uniformly bounded on $[a,b]$ by \cite{Huisken84}. 

On the other hand, by  Theorem \ref{mnf} the following monotonicity formula for $F^{\lambda_{i}}$ holds. 
\begin{align}
\frac{d}{ds}\int_{M^{\lambda_{i}}_{s}}\rho_{0}=-\int_{M^{\lambda_{i}}_{s}}\rho_{0}\biggl\vert\frac{\overrightarrow{F^{\lambda_i}}^{\bot}}{-2s}+H^{\lambda_{i}}\biggr\vert_{\overline{g}}^2.\notag
\end{align}
Integrating the both side of the above equation on any closed interval $[a,b] \subset (-\infty, 0)$, 
we have
\begin{align}
\int_{M^{\lambda_{i}}_{b}}\rho_{0}-\int_{M^{\lambda_{i}}_{a}}\rho_{0}=-\int_{a}^{b}ds\int_{M^{\lambda_{i}}_{s}}\rho_{0}\biggl\vert\frac{\overrightarrow{F^{\lambda_{i}}}^{\bot}}{-2s}+H^{\lambda_{i}}\biggr\vert_{\overline{g}}^2  \label{int}
\end{align}
where we take $i$ sufficiently large so that $[a,b]$ is contained in $[-\lambda_{i}^2T,0)$. 
By Proposition \ref{prop5} we have
\begin{align}
&\int_{M^{\lambda_{i}}_{a}}\rho_{0}=\int_{M_{u_{i}}}\rho_{T}\notag
\end{align}
where $u_{i}=T+a/{\lambda_{i}^{2}}$ and 
\begin{align}
&\int_{M^{\lambda_{i}}_{b}}\rho_{0}=\int_{M_{v_{i}}}\rho_{T}\notag
\end{align}
where $v_{i}=T+b/{\lambda_{i}^{2}}$.  
By the monotonicity formula, the derivative of the function $\int_{M_{t}}\rho_{T}$ is non-positive and $\int_{M_{t}}\rho_{T}\geq0$, 
so for any increasing sequence $\{t_{i}\}_{i=1}^{\infty}$ such that $t_{i}\rightarrow T$ as $i\rightarrow\infty$ 
the sequence $\int_{M_{t_{i}}}\rho_{T}$ converges to  a unique value. 
Now $\{u_{i}\}_{i=1}^{\infty}$ and $\{v_{i}\}_{i=1}^{\infty}$ are increasing sequences such that $u_{i},v_{i}\rightarrow T$ as $i\rightarrow\infty$. 
So $\int_{M^{\lambda_{i}}_{a}}\rho_{0}$ and $\int_{M^{\lambda_{i}}_{b}}\rho_{0}$ converge to the same value as $i\rightarrow\infty$. 
Therefore the left hand side of the equation (\ref{int}) converges to $0$ as $i\rightarrow\infty$, and thus
\begin{align}
\lim_{i\rightarrow\infty}\int_{a}^{b}ds\int_{M^{\lambda_{i}}_{s}}\rho_{0}\biggl\vert\frac{\overrightarrow{F^{\lambda_{i}}}^{\bot}}{-2s}+H^{\lambda_{i}}\biggr\vert_{\overline{g}}^2 = 0.\label{limit}
\end{align}
From this we can take a sequence $s_i \in [a,b]$ such that we have 
\begin{align}
\int_{M^{\lambda_{i}}_{s_i}}\rho_{0}\biggl\vert\frac{\overrightarrow{F^{\lambda_{i}}}^{\bot}}{-2s_i}+H^{\lambda_{i}}\biggr\vert \to 0\label{limit2}
\end{align}
as $i \to \infty$.

Suppose that $p_i$ attains $\min_{M} r(F^{\lambda_i}_{s_i})$, and put
$$ \gamma_i := r^2(F^{\lambda_i}(p_i,s_i )) = \lambda_i^2 r^2(F(p_i, t_i)).$$
Then $p_i$ also attains $\min_{M} r(F_{t_i})$ and 
\begin{equation}\label{below}
 \gamma_i = \lambda_i^2 r^2(F(p_i,t_i)) 
 = \frac{-s_i r^2(F(p_i,t_i)) }{T - t_i}.
 \end{equation}
It then follows from the condition (c) of Definition \ref{typeI_c} that
\begin{equation}\label{below2}
-bK_1 \le \gamma_i \le -aK_2.
 \end{equation}
Thus, the image of $F^{\lambda_i}(\cdot, s_i)$ 
uniformly stays away from the apex, and that $F^{\lambda_i}(p_i, s_i)$ stays in a compact region in $C(N)$ for the minimum point
$(p_i,s_i)$ for $r(F^{\lambda_i})$. 

Put $\gamma := -bK_1$. Let $W$ be the manifold obtained from $C(N)$ by cutting out the portion $\{ r \le \frac{\sqrt\gamma}2\}$, and let $V$ be the manifold 
obtained by gluing $W$ and $-W$ smoothly along their boundaries. This $V$ contains $C(N) - \{r \le \sqrt\gamma\}$ and the image
of $F^{\lambda_i}|_{(M, s_i)}$ is included in that part. 

Since the higher derivatives of the second fundamental form are bounded as shown above,
we can apply Theorem 1.2 in \cite{Cooper10} (see also \cite{ChenHe10}) by taking $(M_k, p_k)$ to be 
$(M, p_k)$, $(N_k, h_k, x_k)$ to be $(V, h, F^{\lambda_k}(p_k, s_k))$ and $F_k$ to be $F^{\lambda_k}$, 
where the metric $h$ is chosen so that $h$ coincides with the
cone metric on $C(N) - \{r \le \sqrt\gamma\}$. 
Then we obtain a limit $F_\infty : M_\infty \to N_\infty$ which satisfies the equation of self-similar solution 
to the mean curvature flow by (\ref{limit2}). 
But since $x_i = F^{\lambda_i}(p_i, s_i)$ stays in a compact region we have $N_\infty = V$. 
The limiting self similar solution then defines a flow in the cone $C(N)$ satisfying the mean curvature equation. 
This completes the proof of Theorem \ref{main}.
\end{proof}

\begin{example}[Examples of type $\mathrm I_c$ singularities.]\label{example}
Here we show a simple example of the mean curvature flow developing the type $I_c$ singularity.
For $-\infty < a < b \le +\infty$, assume that there exists a mean curvature flow $\Phi : M\times [a,b) \to N$ on $(N,g)$, namely $\Phi$ satisfies $\frac{\partial}{\partial s}\Phi(\cdot,s) = H^N_s$, where $H^N_s$ is the mean curvature vector with respect to the embedding $\Phi(\cdot,s):M\to N$.
Then $F:M\times [0,T(1-e^{-2m(b - a)})) \to C(N)$ defined by

\begin{eqnarray}
F(p,t)&:=&(\Phi(p,\alpha(t)), \beta(t)) \in N\times\mathbb{R}^+,\nonumber\\
\alpha(t) &:=& a - \frac{1}{2m}\log (1 - \frac{t}{T})),\nonumber\\
\beta(t) &:=& \sqrt{2m(T - t)},\nonumber
\end{eqnarray}
becomes a solution for mean curvature flow equation with initial data $F_0=\Phi_0:M\to N\times \{\sqrt{2mT}\}\subset C(N)$, where $m = \dim M$.
The second fundamental form ${\rm II}^{C(N)}_t$ of the embedding $F(\cdot,t):M\to C(N)$ is given by
\begin{eqnarray}
{\rm II}^{C(N)}_t &=& {\rm II}^{N}_{\alpha(t)} - r(F(p,t)) g|_{M_t}\otimes\frac{\partial}{\partial r},\nonumber
\end{eqnarray}
where ${\rm II}^{N}_{\alpha (t)}$ is the second fundamental form of the embedding $M_t = \Phi(M,\alpha (t))\subset N$.
Then we obtain
\begin{eqnarray}
|{\rm II}^{C(N)}_t|_{\bar{g}}^2 \le \frac{m}{2(T - t)} ( 1 + \frac{1}{m^2}\sup_{p\in M}|{\rm II}^{N}_{\alpha(t)}(p)|_g^2),\nonumber
\end{eqnarray}
since $| {\rm II}^{N}_{\alpha(t)} |_{\bar{g}} = r(F(p,t))^{-1}|{\rm II}^{N}_{\alpha(t)}|_g$.
Hence $F$ develops a type $I$ singularity at $t=T$, if $b = +\infty$ and
\begin{eqnarray}
\sup_{p\in M,s \ge a}|{\rm II}^{N}_{s}(p)|_g < \infty.\nonumber
\end{eqnarray}
The condition (b) and (c) of Definition \ref{typeI_c} are obviously satisfied since $r ( F (p,t)) = \sqrt{2m(T - t)}$.

\end{example}

\section{Special Lagrangian submanifolds in toric Calabi-Yau cones}

In this section we construct special Lagrangian submanifolds in toric Calabi-Yau cones. 
Let $V$ be a Ricci-flat K\"ahler manifold with a K\"ahler form $\omega$ and of $\dim_\bfC V = n$.
Then the canonical line bundle $K_V$ is flat. $V$ is said to be a Calabi-Yau manifold if in addition $K_V$ is trivial
and $V$ admits a parallel holomorphic $n$-form $\Omega$. This implies that, with a suitable normalization of $\Omega$, we have
$$ \frac{\omega^n}{n!} =  (-1)^{\frac{n(n-1)}2} \biggl(\frac{\sqrt{-1}}2\biggr)^n \Omega \wedge \overline{\Omega}.$$
Let $L$ be a real oriented $n$-dimensional submanifold of $V$. Then $L$ is called a special Lagrangian submanifold of $V$ if
$\omega|_L = 0$ and $\mathrm{Im}\Omega|_L = 0$.

Toric Calabi-Yau cones are exactly the K\"ahler cones over Sasaki-Einstein manifolds. They are described as toric K\"ahler cones
obtained from toric diagram of height $1$. This result was obtained in \cite{FOW} and \cite{CFO}, which we outline below.

\begin{definition}[Good rational polyhedral cones, c.f. \cite{Lerman}] \label{good} Let $\mathfrak g^{\ast}$ be the dual of the Lie algebra $\mathfrak g$
of an $n$-dimensional torus $G$. Let $\bfZ_{\mathfrak g}$ be the integral lattice of $\mathfrak g$, that is
the kernel of the exponential map $\exp : \mathfrak g \to G$. 
A subset $C \subset \mathfrak g^{\ast}$ is a rational polyhedral cone if there exists a
finite set of vectors $\lambda_i \in \bfZ_{\mathfrak g}$, $1 \le i \le d$, such that
$$ C = \{ y \in \mathfrak g^{\ast}\ |\ \la y, \lambda_i \ra \ge 0\ \mathrm{for\ }\ i = 1, \cdots, d\}.$$
We assume that the set $\lambda_i$ is minimal in that for any $j$
$$ C \ne \{ y \in \mathfrak g^{\ast}\ |\ \la y, \lambda_i \ra \ge 0\ \mathrm{for\ all}\ i\ne j\}$$
and that each $\lambda_i$ is primitive, i.e. $\lambda_i$ is not of the form $\lambda_i = a\mu$
for an integer $a \ge 2$ and $\mu \in \bfZ_{\mathfrak g}$. 
(Thus $d$ is
the number of facets if $C$ has non-empty interior.)
Under these two assumptions
a rational polyhedral cone $C$ with nonempty interior is said to be good if the following condition holds.
If
$$ \{ y \in C\ |\ \la y, \lambda_{i_j}\ra = 0\ \mathrm{for\ all}\ j = 1, \cdots, k \}$$
is a non-empty face of $C$ for some $\{i_1, \cdots, i_k\} \subset \{1, \cdots, d\}$, then 
$\lambda_{i_1}, \cdots, \lambda_{i_k}$ are linearly independent
over $\bfZ$ and generates the subgroup
$\{ \sum_{j=1}^k a_j \lambda_{i_j}\ |\ a_j \in \bfR\} \cap \bfZ_{\mathfrak g} $.
\end{definition}

\begin{definition}[Toric diagrams of height $\ell$, c.f. \cite{CFO}] \label{TD1} An $n$-dimensional toric diagram with height $\ell$ is a
collection of $\lambda_i \in \bfZ^{n} \cong 
\bfZ_{\mathfrak g}$ which define a good rational polyhedral cone and 
$\gamma  \in \bfQ^{n} \cong
(\bfQ_{\mathfrak g})^{\ast}$
 such that 
\begin{enumerate}
\item[(1)] $\ell$ is a positive integer such that $\ell\gamma$ is a primitive element of
the integer lattice $\bfZ^{n} \cong \bfZ^{\ast}_{\mathfrak g}$.
\item[(2)] $\la \gamma, \lambda_i\ra = -1$. 
\end{enumerate}
We say that a good rational polyhedral
cone $C$ is associated with a toric diagram of height $\ell$ if there exists a rational vector $\gamma$ 
satisfying $(1)$ and $(2)$ above.
\end{definition}
The reason why we use the terminology ``height $\ell$'' is because using a transformation by an element of $SL(n,\bfZ)$ we may assume that
$$ \gamma = \left(\begin{array}{r} -\frac1\ell \\ 0 \\ \vdots \\ 0 \end{array}\right)$$
and the first component of $\lambda_i$ is equal to $\ell$ for each $i$.

\begin{theorem}[\cite{FOW}, \cite{CFO}] \label{toricSE} Toric Sasaki-Einstein manifolds are exactly those whose K\"ahler cones are obtained 
by the Delzant construction from toric diagram of fixed height and applying the volume minimization of
Martelli-Sparks-Yau \cite{MSY2}. Equivalently, Toric Ricci-flat K\"ahler manifolds are exactly those obtained 
by the Delzant construction from toric diagram of fixed height and applying the volume minimization of
Martelli-Sparks-Yau \cite{MSY2}.
\end{theorem}
\noindent
 For a Ricci-flat toric K\"ahler cone $V$ obtained from a toric diagram of height $\ell$, there
exists a parallel holomorphic section of $K_V^{\otimes\ell}$. In particular if $\ell = 1$ the K\"ahler cone manifold $V$ is a Calabi-Yau manifold.
From now on we assume $\ell =1$. Then it is shown in \cite{CFO} that the parallel holomorphic $n$-form is given in the form
$$ \Omega = e^{-\sum_{i=1}^n \gamma_i z^i} dz^1 \wedge \cdots \wedge dz^n $$
where $z^1, \cdots, z^n$ are holomorphic logarithmic coordinates. Since $V$ is obtained from a toric diagram of height $1$ we 
may assume $\gamma = {}^t(-1, 0, \cdots, 0)$. In this case we have
$$ \Omega = e^{z^1}dz^1 \wedge \cdots \wedge dz^n.$$

We now apply a method used in \cite{IonelMinoo08} and \cite{Kawai11}. Their method is summarized in \cite{Kawai11} as follows.
\begin{proposition}[\cite{Kawai11}]\label{Kawai}
Let $(V,J,\omega, \Omega)$ be a Calabi-Yau manifold of complex dimension $n$, and $H$ be a compact connected Lie group of
real dimension $n-1$ acting effectively on $V$ preserving the Calabi-Yau structure. Suppose there exist a moment map $\mu : V \to \frak h^\ast$
and a $H$-invariant $(n-1)$-form $\alpha$ such that for any $X_1, \cdots, X_{n-1} \in \frak h$ we have
$$ \mathrm{Im} \Omega (\cdot,X_1, \cdots, X_{n-1}) = d(\alpha(X_1, \cdots, X_{n-1}))$$
where $X_i \in \frak h$ are identified with vector fields on $V$. Then for any $c \in Z_{\frak h^\ast}$, $c' \in \bfR$ and any basis $\{Y_1, \cdots, Y_{n-1}\}
 \subset \frak h$, the set
 $$ L_{c,\,c'} = \mu^{-1}(c) \cap (\alpha(Y_1, \cdots, Y_{n-1}))^{-1}(c') $$
 is a $H$-invariant special Lagrangian submanifold of $V$. 
\end{proposition}
\noindent
We refer the reader to \cite{Kawai11} for the proof of Proposition \ref{Kawai}. We now apply Proposition \ref{Kawai} to toric Calabi-Yau 
manifold obtained from toric diagrams of height $1$ with
$$ \Omega = e^{z^1}dz^1 \wedge \cdots \wedge dz^n, \quad\quad \alpha = \mathrm{Im} (e^{z^1}dz^2 \wedge \cdots \wedge dz^n),$$
and with $Y_j = 2\mathrm{Im} (\frac{\partial}{\partial z^j})$ and $H$ the subtorus $T^{n-1}$ generated by $Y_1, \cdots, Y_{n-1}$. Then one easily finds that
$$ \mathrm{Im} \Omega (\cdot,Y_1, \cdots, Y_{n-1}) = d(\alpha(Y_1, \cdots, Y_{n-1})), $$
and
$$ \alpha(Y_1, \cdots, Y_{n-1}) = \frac 1{ i^n} (e^{z^1} + (-1)^n e^{\overline{z_1}}). $$
Thus the assumptions of Proposition \ref{Kawai} is satisfied, and we have proved the following.
\begin{theorem}\label{SLagtoric} Let $V$ be a toric Calabi-Yau manifold obtained from a toric diagram of height $1$. Let 
$$ \Omega = e^{z^1} dz^1 \wedge \cdots \wedge dz^n $$
be the parallel holomorphic $n$-form described as above. Then there is a $T^{n-1}$-invariant special Lagrangian submanifold
described as 
$$\mu^{-1}(c) \cap \{(e^{z^1} + (-1)^n e^{\overline{z_1}})/i^n = c'\}$$ 
where $T^{n-1}$ is a subtorus generated by
$\mathrm{Im}(\partial/\partial z^2), \cdots, \mathrm{Im}(\partial/\partial z^n)$ and $\mu : V \to \frak h^\ast$ is a moment map.
\end{theorem}

\begin{example} Take $V$ to be the flat $\bfC^n$, and let $w^1, \cdots, w^n$ be the standard holomorphic coordinates with
$$ \Omega = dw^1 \wedge \cdots \wedge dw^n.$$
The logarithmic holomorphic coordinates $v^1, \cdots , v^n$ are given by $w^i = e^{v^i}$. Thus, we have
$$ \Omega = e^{(v^1 + \cdots + v^n)}dv^1 \wedge \cdots \wedge dv^n.$$
Taking $\gamma = {}^t(-1, 0, \cdots, 0)$ amounts to changing the coordinates $z^1 = v^1 + \cdots + v^n$, $z^2 = v^2$, $\cdots$, $z^n = v^n$.
Then with the new coordinates we have
$$ \Omega = e^{z^1} dz^1 \wedge \cdots \wedge dz^n.$$
In this situation the points in $\mu^{-1}(c)$ are described as
$$ |w^2|^2 - |w^1|^2 = c_2,\ \cdots,\ |w^n|^2-|w^1|^2 = c_n.$$
If $n$ is even then $(e^{z^1} + (-1)^n e^{\overline{z_1}})/i^n = c'$ if and only if $\mathrm{Re} (w^1 \cdots w^n) = c'$, and 
If $n$ is odd then $(e^{z^1} + (-1)^n e^{\overline{z_1}})/i^n = c'$ if and only if $\mathrm{Im} (w^1 \cdots w^n) = c'$.
This is exactly the same as Theorem 3.1 in \cite{HarveyLawson82Acta}.
\end{example}

\section{Infinitesimal deformations of special Lagrangian cones}
In this section we consider the infinitesimal deformations of special Lagrangian cones embedded in the cone of Sasaki-Einstein manifolds.

\begin{definition}
{\rm A Riemannian manifold $(N,g)$ is called a} Sasakian manifold {\rm if its Riemannian cone $(C(N),\bar{g})$ is a K\"ahler manifold with respect to some integrable complex structure $J$ over $C(N)$. A} Reeb vector field {\rm $\xi$ on the Sasakian manifold $(N,g)$ is a Killing vector field on $N$ given by $\xi:=J(r\frac{\partial}{\partial r})$.}
\end{definition}

For a Sasakian manifold $(N,g)$, a contact form $\eta \in \Omega^1(N)$ on $N$ is given by $\eta:=g(\xi,\cdot)$. Then the K\"ahler form $\omega\in \Omega^2(C(N))$ on $C(N)$ is described as $\omega = d(r^2\eta)$.

\begin{definition}
{\rm For a smooth manifold $N$, a} cone submanifold $C$ of $C(N)$ {\rm is a submanifold of $C(N)$ which can be written as $C=C(\Sigma)$ for a submanifold $\Sigma \subset N$. For a Sasakian manifold $(N,g,\xi)$, a cone submanifold $C\subset C(N)$ is a} Lagrangian cone {\rm if it is a Lagrangian submanifold of $(C(N),\omega)$. }
\end{definition}

The following proposition is well-known but here we give a proof for readers' convenience.

\begin{proposition}\label{legendrian}
A submanifold $\Sigma \subset N$ is Legendrian if and only if $C(\Sigma) = \Sigma \times \mathbb{R}^+ \subset C(N)$ is Lagrangian with respect to the K\"aler form $\omega$ on $C(N)$.
\end{proposition}

\begin{proof}
Let $\Sigma\subset N$ be a Legendrian submanifold. For any $p\in \Sigma$, open neighborhood $U\subset \Sigma$ and $u,v\in \mathcal{X}(U)$, we have
\begin{eqnarray}
\omega(u,v) &=& d\eta (u,v) = -\eta ([u,v]) = 0,\nonumber\\
\omega(u,\frac{\partial}{\partial r}) &=& g(u,\xi) = \eta(u) = 0,\nonumber
\end{eqnarray}
since $\eta|_\Sigma = 0$ and $[u,v]\in \mathcal{X}(U)$. Hence $C(\Sigma)\subset C(N)$ is Lagrangian. Conversely, let $C(\Sigma)\subset C(N)$ be Lagrangian and take $u\in T_p\Sigma$ arbitrarily. Then
\begin{eqnarray}
\eta(u) = g(u,\xi) = \omega(u,\frac{\partial}{\partial r}) = 0,\nonumber
\end{eqnarray}
which implies that $\Sigma\subset N$ is a Legendrian submanifold.
\end{proof}

\begin{proposition}\label{hol.vol.}
Let $(V,J,\omega)$ be a Ricci-flat K\"ahler manifold of $\dim_{\mathbb{C}} = n$ with $H^1_{DR}(V,\mathbb{R}) = 0$, and assume that the canonical line bundle $K_{V}$ is holomorphically trivial. Then there exists a holomorphic $n$ form $\Omega\in\Omega^{(n,0)}(V)$ satisfying
\begin{eqnarray}
\frac{\omega^n}{n!} = (-1)^{\frac{n(n-1)}{2}}\bigg( \frac{\sqrt{-1}}{2} \bigg)^n\Omega\wedge \overline{\Omega}.\label{normalization}
\end{eqnarray}
\end{proposition}
\begin{proof}
From the assumption there is a nowhere vanishing holomorphic $n$ form $\Omega_0\in\Omega^{(n,0)}(V)$ on $V$.
Since $\Omega_0$ is holomorphic, $d\Omega_0 = 0$.
The K\"ahler form $\omega$ on $V$ induces a hermitian metric on $K_{V}$ by
\begin{eqnarray}
h := h(\Omega_0, \overline{\Omega_0}) := n!(-1)^{\frac{n(n-1)}{2}}\bigg( \frac{\sqrt{-1}}{2} \bigg)^n \frac{\Omega_0\wedge \overline{\Omega_0}}{\omega^n}.\nonumber
\end{eqnarray}
Now we put $\Omega := h^{-\frac{1}{2}}e^{\sqrt{-1}\rho}\Omega_0$ for $\rho\in C^{\infty}(V,\mathbb{R})$, which satisfies the equation (\ref{normalization}). 
Then it suffices to show that there exists $\rho\in C^{\infty}(V,\mathbb{R})$ such that $d\Omega =0$.
From $d\Omega_0 = 0$, we have
\begin{eqnarray}
d\Omega &=& \bar{\partial}(h^{-\frac{1}{2}}e^{\sqrt{-1}\rho})\wedge\Omega_0\nonumber\\
&=& h^{-\frac{1}{2}}e^{\sqrt{-1}\rho}(-\frac{1}{2}h^{-1}\bar{\partial} h + \sqrt{-1} \bar{\partial} \rho)\wedge\Omega_0\nonumber\\
&=& h^{-\frac{1}{2}}e^{\sqrt{-1}\rho}(-\frac{1}{2}\bar{\partial} \log h + \sqrt{-1} \bar{\partial} \rho)\wedge\Omega_0.\nonumber
\end{eqnarray}
Thus the problem is reduced to show the existence of the function $\rho$ which satisfies $\bar{\partial} ( -\frac{1}{2} \log h + \sqrt{-1}  \rho) = 0$.

Recall that $\omega$ is Ricci-flat K\"ahler form. Then the curvature form of the Hermitian connection on $K_V$ induced from $h$ is equal to zero, we have $dd^c \log h = 0$. Now we have assumed $H^1_{DR}(V,\mathbb{R}) = 0$, there exists $\hat{\rho}\in C^{\infty}(V,\mathbb{R})$ such that $d^c \log h = (\sqrt{-1}\partial - \sqrt{-1}\bar{\partial}) \log h = d \hat{\rho} = (\partial + \bar{\partial}) \hat{\rho}$. By comparing $(0,1)$-part, we have $\bar{\partial} (\log h - \sqrt{-1}\hat{\rho}) = 0$, consequently we obtain the assertion by putting $\hat{\rho} = 2\rho$.
\end{proof}
From now on suppose $(N,g,\xi)$ is a Sasaki-Einstein manifold of dimension $2n-1$, hence the K\"ahler structure $\omega$ on $C(N)$ is Ricci-flat.
Moreover we assume the canonical bundle $K_{C(N)}$ is trivial.
Since $(N,g)$ is an Einstein manifold with positive Ricci curvature, then $H^1(C(N),\mathbb{R}) = H^1(N,\mathbb{R}) = 0$. Therefore we have a holomorphic $n$-form $\Omega$ on $C(N)$ satisfying (\ref{normalization}).

Now we denote by $\tilde{H}$ and $H$ the mean curvature vector of $C(\Sigma)\subset C(N)$ and $\Sigma\subset N$, respectively. Then the direct calculation gives $\tilde{H} = r^{-2}H$, therefore $C(\Sigma)$ is minimal if and only if $\Sigma$ is minimal.

It is well known that the mean curvature of a Lagrangian submanifold embedded in a Calabi-Yau manifold is equal to $d\theta$ under the identification of vector fields and $1$-forms by the symplectic form, where $\theta$ is the Lagrangian angle. 
Then the Lagrangian submanifold embedded in the Calabi-Yau manifold is minimal if and only if the Lagrangian angle is constant. 
In particular it is special Lagrangian if the Lagrangian angle is equal to zero.
Hence $\Sigma\subset N$ is minimal Legendrian if and only if $C(\Sigma)\subset C(N)$ is Lagrangian with constant Lagrangian angle.

In \cite{Ohnita}, the infinitesimal deformation spaces of minimal Legendrian submanifolds embedded in $\eta$-Sasaki-Einstein manifolds are studied.
Here we observe the infinitesimal deformation spaces of special Lagrangian cones in $C(N)$, using the results obtained in \cite{McLean}.

Let $C(\Sigma)$ be a special Lagrangian submanifold in $C(N)$, and we have orthogonal decompositions $TC(N)|_{C(\Sigma)} = TC(\Sigma) \oplus NC(\Sigma)$ and $TN|_\Sigma = T\Sigma \oplus N\Sigma$, where $N\Sigma,NC(\Sigma)$ are normal bundles. 
Then for any $(x,r)\in C(\Sigma)$ we have the natural identification $N_{(x,r)}C(\Sigma) = N_x\Sigma$.

The infinitesimal deformations of cone submanifolds of $C(N)$ is generated by the smooth $1$-parameter families of cone submanifolds $\{C(\Sigma_t) = {\pi_N}^{-1}(\Sigma_t);\ -\varepsilon < t <\varepsilon \}$, where $\{\Sigma_t;\ -\varepsilon < t <\varepsilon \}$ is the smooth families of submanifolds of $N$ which satisfies $\Sigma_0=\Sigma$, and $\pi_N:N\times\mathbb{R}^+ \to N$ is the projection onto the first component.
Since the infinitesimal deformations of $\Sigma \subset N$ are parameterized by smooth sections of $N\Sigma$, the infinitesimal deformations of cone submanifolds are parameterized by
\begin{eqnarray}
\mathcal{A}_{C(\Sigma)} &:=& \{ \alpha = {\pi_N}^*\alpha_0 \in \Gamma (NC(\Sigma)) ; \ \alpha_0\in \Gamma (N\Sigma)\}.\nonumber
\end{eqnarray}
Then $\alpha_{(x,r)}\in N_{(x,r)}C(\Sigma) = N_x\Sigma$ is independent of $r$ for each $\alpha \in \mathcal{A}_{C(\Sigma)}$.

Since $C(\Sigma)$ is Lagrangian, $NC(\Sigma)$ is identified with the cotangent bundle $T^*C(\Sigma)$ by the bundle isomorphism $\hat{\omega}:NC(\Sigma) \to T^*C(\Sigma)$ defined by $\hat{\omega}(v) := \iota_v\hat{\omega} = \hat{\omega} (v,\cdot)$.

By the results in \cite{McLean}, the infinitesimal deformations of special Lagrangian submanifolds of $C(\Sigma)$ are parameterized by harmonic $1$-forms on $C(\Sigma)$.
Thus the infinitesimal deformations of special Lagrangian cones of $C(\Sigma)$ are parameterized by
\begin{eqnarray}
\mathcal{H}_{C(\Sigma)} := \{\hat{\omega} (\alpha) \in \Omega^1(C(\Sigma));\ \alpha \in \mathcal{A}_{C(\Sigma)},\ d\hat{\omega} (\alpha) = d*\hat{\omega} (\alpha) = 0\},\nonumber
\end{eqnarray}
where $\hat\omega$ is the isomorphism induced by $\omega$, and $*$ is the Hodge star with respect to the induced metric $\bar{g}|_{C(\Sigma)}$.
To study the vector space $\mathcal{H}_{C(\Sigma)}$, we need the next lemma.
\begin{lemma}\label{normal.cot.}
Under the natural identification $T^*_{(x,r)}C(\Sigma) = T^*_x\Sigma \oplus T^*_r\mathbb{R}^+$, we have
\begin{eqnarray}
\hat{\omega} (\mathcal{A}_{C(\Sigma)}) = \{ \beta_{(x,r)} = r\varphi (x)dr + r^2 \gamma_x \in \Omega^1(C(\Sigma));\ \varphi \in C^{\infty}(\Sigma),\ \gamma \in \Omega^1(\Sigma) \}.\nonumber
\end{eqnarray}
\end{lemma}
\begin{proof}
Define a diffeomorphism $m_a = \exp (ar\frac{\partial}{\partial r}):C(N)\to C(N)$ by $m_a(p,r) = (p,ar)$ for $a>0$.
First of all we show that $m_a$ is a biholomorphism.
Since $\frac{d}{da}(m_a)_*J = (m_a)_*\mathcal{L}_{r\frac{\partial}{\partial r}}J$, it suffices to show $\mathcal{L}_{r\frac{\partial}{\partial r}}J = 0$. Now we may write $r\frac{\partial}{\partial r} = -J\xi$, then for any $x\in C(N)$ and open neighborhood $x\in U\subset C(N)$ and $v\in\mathcal{X}(C(N))$,
\begin{eqnarray}
(\mathcal{L}_{J\xi}J)(v) &=& \mathcal{L}_{J\xi}(Jv) - J(\mathcal{L}_{J\xi}v)\nonumber\\
&=& [J\xi,Jv] - J([J\xi,v])\nonumber\\
&=& -N_J(\xi,v) - J^2[\xi,v] + J[\xi,Jv]\nonumber\\
&=& -N_J(\xi,v) + J\{(\mathcal{L}_{\xi}J)(v)\},\nonumber
\end{eqnarray}
where $N_J$ is the Nijenhuis tensor. Thus we have $\mathcal{L}_{J\xi}J = 0$ since $J$ is integrable and $\mathcal{L}_{\xi}J = 0$, hence $m_a$ is a biholomorphism.

Next we show that 
\begin{eqnarray}
\hat{\omega} (\mathcal{A}_{C(\Sigma)}) = \{ \beta\in \Omega^1(C(\Sigma)) ;\ {m_a}^*\beta = a^2\beta\ {\rm for\ all\ }a\in\mathbb{R}^+\}.\nonumber
\end{eqnarray}
Since $m_a$ satisfies ${m_a}^*\bar{g} = {m_a}^*(dr^2 + r^2g) = d(ar)^2 + (ar)^2g = a^2\bar{g}$, we obtain ${m_a}^*\omega = a^2\omega$.
By the definition of $\mathcal{A}_{C(\Sigma)}$, we may write
\begin{eqnarray}
\mathcal{A}_{C(\Sigma)} = \{ \alpha\in \Omega^1(C(\Sigma)) ;\ (m_a)_*\alpha = \alpha \ {\rm for\ all\ }a\in\mathbb{R}^+\}.\nonumber
\end{eqnarray}
For any $\alpha \in \Gamma(NC(\Sigma))$, we have
\begin{eqnarray}
m_a^*(\hat{\omega}(\alpha)) &=& m_a^*(\iota_{\alpha}\omega) = \iota_{({m_a})_*^{-1}\alpha}m_a^*\omega = a^2\hat{\omega} (({m_a})_*^{-1}\alpha)\nonumber\\
&=& a^2 \hat{\omega}(\alpha) + a^2\hat{\omega} (({m_a})_*^{-1}\alpha - \alpha).\nonumber
\end{eqnarray}
Therefore the equation $m_a^*(\hat{\omega}(\alpha)) = a^2 \hat{\omega}(\alpha) $ holds for all $a\in\mathbb{R}^+$ if and only if $\alpha\in\mathcal{A}_{C(\Sigma)}$.

Now we take $\beta\in\Omega^1(C(\Sigma))$ and decompose it as $\beta_{(x,r)} = \sigma(x,r) + \tau(x,r)dr$ such that $\sigma(x,r)\in T^*_x\Sigma$ and $\tau\in C^{\infty}(C(\Sigma))$. 
\begin{eqnarray}
{m_a}^*\beta = {m_a}^*\sigma + {m_a}^*\tau \cdot adr,\nonumber
\end{eqnarray}
then ${m_a}^*\beta = a^2\beta$ is equivalent to
\begin{eqnarray}
\sigma(x,ar) &=& a^2\sigma(x,r),\nonumber\\
\tau (x,ar) &=& a \tau (x,r).\nonumber
\end{eqnarray}
Thus we may put $\sigma = r^2\gamma$ and $\tau = r\varphi$ for some $\gamma\in\Omega^1(\Sigma)$ and $\varphi\in C^{\infty}(\Sigma)$.
\end{proof}

\begin{theorem}\label{infin.deform.sp}
The vector space $\mathcal{H}_{C(\Sigma)}$ is isomorphic to
\begin{eqnarray}
{\rm Ker}(\Delta_\Sigma - 2n) = \{\varphi \in C^{\infty}(\Sigma);\ \Delta_\Sigma \varphi = 2n \varphi \},\nonumber
\end{eqnarray}
where $\Delta_\Sigma = d^{*_{\Sigma}}d$ and $d^{*_{\Sigma}}$ is a formal adjoint operator of $d$ with respect to the metric $\bar{g}|_{\Sigma}$.
\end{theorem}
\begin{proof}
From Lemma \ref{normal.cot.}, all $\beta\in \hat{\omega} (\mathcal{A}_{C(\Sigma)})$ can be written as $\beta = r\varphi dr + r^2 \gamma$. Then we have
\begin{eqnarray}
d\beta = rdr\wedge (2\gamma -d\varphi) + r^2d\gamma,\nonumber
\end{eqnarray}
from which it follows that $d\beta = 0$ is equivalent to $2\gamma = d\varphi$.

Next we calculate $d*\beta$.
Denote by ${\rm vol}_{\Sigma}$ the volume form of $g|_{\Sigma}$. Since the volume forms of $\bar{g}|_{C(\Sigma)}$ is given by $r^{n-1}dr\wedge {\rm vol}_{\Sigma}$, we can deduce
\begin{eqnarray}
*\gamma &=& -r^{n-3}dr \wedge *_{\Sigma}\gamma,\nonumber\\
*dr &=& r^{n-1}{\rm vol}_{\Sigma},\nonumber
\end{eqnarray}
where $*_{\Sigma}$ is the Hodge star operator with respect to $g|_{\Sigma}$.
Consequently, we obtain
\begin{eqnarray}
d*\beta = r^{n-1}dr \wedge (d*_{\Sigma}\gamma + n\varphi{\rm vol}_{\Sigma}).\nonumber
\end{eqnarray}
Hence $d\beta = d*\beta = 0$ is equivalent to
\begin{eqnarray}
\gamma = \frac{1}{2}d\varphi,\quad n\varphi {\rm vol}_\Sigma + \frac{1}{2}d*_\Sigma d\varphi = 0,\nonumber
\end{eqnarray}
and the latter equation is equivalent to $d^{*_{\Sigma}}d\varphi = 2n \varphi$.
\end{proof}

In \cite{Ohnita}, the infinitesimal deformation spaces of minimal Legendrian submanifolds in Sasaki-Einstein manifolds are studied.
Proposition \ref{infin.deform.sp} is also obtained from the case of $\eta$-Ricci constant $A$ is  equal to $2n-2$ in \cite{Ohnita}.
Here we should pay attention that the dimension of infinitesimal deformation spaces obtained in \cite{Ohnita} is equal to $1 + \dim {\rm Ker}(\Delta_\Sigma - 2n)$, 
since the deformations of $C(\Sigma)$ generated by Reeb vector field $\xi$ is not special Lagrangian cone, 
but minimal Lagrangian cone whose Lagrangian angle is not equal to zero.
Actually, if we put $\alpha = \xi$, then $\beta = \hat{\omega}(\alpha) = -rdr$ and
\begin{eqnarray}
d*\beta = -d(r^n {\rm vol}_\Sigma ) = -nr^{n-1}dr\wedge {\rm vol}_\Sigma \neq 0,\nonumber
\end{eqnarray}
accordingly this $\alpha$ does not generate deformations of special Lagrangian cones.

\bibliographystyle{amsalpha}

\end{document}